\documentclass[reqno,12pt]{amsart}
\usepackage{amscd,amsfonts,amssymb}
\numberwithin{equation}{section}
\newtheorem{Theorem}{Theorem}[section]
\newtheorem{Lemma}{Lemma}[section]

\theoremstyle{definition}

\theoremstyle{remark}
\newtheorem{Remark}{Remark}[section]
\newtheorem{Example}{Example}[section]
\newcommand{\sign}{\mathop{\rm sign}\nolimits}

\newcommand{\essinf}{\mathop{\rm ess \, inf}\limits}

\author{A.A. Kon'kov}
\address{Department of Differential Equations,
Faculty of Mechanics and Mathematics,
Mo\-s\-cow Lo\-mo\-no\-sov State University,
Vorobyovy Gory,
Moscow, 119992 Russia}
\email{konkov@mech.math.msu.su}
\author{A.E. Shishkov}
\address{
RUDN University,
Miklukho-Maklaya str. 6,
Moscow, 117198 Russia;
Institute of Applied Mathematics and Mechanics of NAS of Ukraine,
Dobrovol'skogo str. 1, Slavyansk, 84116 Ukraine
}
\email{aeshkv@yahoo.com}
\thanks{The work of the second author is supported by the ``RUDN University Program 5--100''}
\title[On blow-up conditions]{On blow-up conditions for solutions of higher order differential inequalities}
\thanks{}
\date{}

\begin{document}

\begin{abstract}
For differential inequalities of the form
$$
	\sum_{|\alpha| = m}
	(- 1)^m
	\partial^\alpha
	a_\alpha (x, u)
	\ge
	b (x) |u|^\lambda
	\quad
	\mbox{in } {\mathbb R}^n,
	\:
	n \ge 1,
$$
where $a_\alpha$ and $b$ are some functions,
we obtain conditions guaranteeing that any solution is identically equal to zero.
We construct examples which show that the obtained conditions are sharp.
\end{abstract}

\maketitle

\section{Introduction}
We study solutions of the inequality
\begin{equation}
	\sum_{|\alpha| = m}
	(- 1)^m
	\partial^\alpha
	a_\alpha (x, u)
	\ge
	b (x) |u|^\lambda
	\quad
	\mbox{in } {\mathbb R}^n,
	\label{1.1}
\end{equation}
where $n \ge 1$ and $m \ge 1$ are some integers and $\lambda$ is a real number.
It is assumed that $b$ is a positive measurable function and
\begin{equation}
	|a_\alpha (x, \zeta)| \le a (x) |\zeta|
	\label{1.2}
\end{equation}
with some positive measurable function $a$ 
for almost all $x = (x_1, \ldots, x_n) \in {\mathbb R}^n$ 
and for all $\zeta \in {\mathbb R}$ and $|\alpha| = m$.
As is customary, by $\alpha = (\alpha_1, \ldots, \alpha_n)$ we mean a multi-index.
In so doing,
$|\alpha| = \alpha_1 + \ldots + \alpha_n$
and
$\partial^\alpha = \partial^{|\alpha|} / \partial_{x_1}^{\alpha_1} \ldots \partial_{x_n}^{\alpha_n}$.
Let us also denote by $B_r$ an open ball in ${\mathbb R}^n$ of radius $r > 0$ centered at zero.

A function $u$ is called a solution of inequality~\eqref{1.1} if
$b (x) |u|^\lambda \in L_{1, loc} ({\mathbb R}^n)$
and
$a_\alpha (x, u) \in L_{1, loc} ({\mathbb R}^n)$ for all $|\alpha| = m$ and, moreover,
\begin{equation}
	\int_{{\mathbb R}^n}
	\sum_{|\alpha| = m}
	a_\alpha (x, u)
	\partial^\alpha
	\varphi
	\,
	dx
	\ge
	\int_{{\mathbb R}^n}
	b (x) |u|^\lambda
	\varphi
	\,
	dx
	\label{1.3}
\end{equation}
for any non-negative function $\varphi \in C_0^\infty ({\mathbb R}^n)$.

Our aim is to obtain conditions guaranteeing that every solution of~\eqref{1.1} is trivial. 
Questions treated in our paper were studied earlier by a number of authors
mainly for inequalities of the second order~[1--11, 13--18].
Higher order inequalities have been studied much less.
Until recently, the only effective method for their study was the method described 
in~\cite{GMP, MPbook}.
However, this method loses its exactness in the case where the function $1 / b$ performs 
large oscillations in a neighborhood of infinity.
In particular, it can not be applied to inequalities considered in Example~\ref{E2.2}.

\section{Main results}

\begin{Theorem}\label{T2.1}
Let $\lambda > 1$ and
\begin{equation}
	\int_1^\infty
	r^{(m - n) \lambda + n - 1}
	q (r)
	\,
	dr
	=
	\infty,
	\label{T2.1.1}
\end{equation}
where
$$
	q (r)
	=
	\essinf_{
		B_{\sigma r}
		\setminus
		B_{r / \sigma} 
	}
	a^{- \lambda} b
$$
for some real number $\sigma > 1$.
Then any solution of~\eqref{1.1} is trivial, i.e. $u (x) = 0$ for almost all $x \in {\mathbb R}^n$.
\end{Theorem}

\begin{Theorem}\label{T2.2}
Let $\lambda > 1$ and~\eqref{T2.1.1} be valid, where
\begin{equation}
	q (r)
	=
	\left(
		\frac{1}{r^n}
		\int_{
			B_{\sigma r}
			\setminus
			B_{r / \sigma} 
		}
		a^{\lambda / (\lambda - 1)} (x)
		b^{- 1 / (\lambda - 1)} (x)
		\,
		dx
	\right)^{1 - \lambda}
	\label{T2.2.1}
\end{equation}
for some real number $\sigma > 1$.
Then any solution of~\eqref{1.1} is trivial.
\end{Theorem}

\begin{Remark}\label{R2.1} In relation~\eqref{T2.2.1}, we do not exclude that
$$
	\int_{
		B_{\sigma r}
		\setminus
		B_{r / \sigma} 
	}
	a^{\lambda / (\lambda - 1)} (x)
	b^{- 1 / (\lambda - 1)} (x)
	\,
	dx
	=
	\infty
$$
for some real numbers $r > 0$. Since $\lambda > 1$, in this case, we have $q (r) = 0$.
\end{Remark}

Theorems~\ref{T2.1} and~\ref{T2.2} are proved in Section~\ref{proofOfTheorems}.

\begin{Example}\label{E2.1}
Consider the inequality
\begin{equation}
	(- 1)^k \Delta^k u \ge b (x) |u|^\lambda
	\quad
	\mbox{in } {\mathbb R}^n,
	\label{E2.1.1}
\end{equation}
where $\lambda > 1$ and $b$ is a positive measurable function such that
\begin{equation}
	b (x) \ge b_0 |x|^l,
	\quad
	b_0 = const > 0,
	\label{E2.1.2}
\end{equation}
for almost all $x$ from a neighborhood of infinity.
According to Theorem~\ref{T2.1}, if
\begin{equation}
	l \ge (n - 2 k) \lambda - n,
	\label{E2.1.3}
\end{equation}
then any solution of~\eqref{E2.1.1} is trivial.
It can be seen that~\eqref{E2.1.3} coincides with the well-known blow-up condition
given in~\cite[Example~5.2]{MPbook}.

Now, assume that, in a neighborhood of infinity, the function $b$ satisfy the inequality
\begin{equation}
	b (x) 
	\ge 
	b_0 
	|x|^{(n - 2 k) \lambda - n}
	\log^\nu |x|,
	\quad
	b_0 = const > 0.
	\label{E2.1.4}
\end{equation}
In other words, we examine the critical exponent $l = (n - 2 k) \lambda - n$ in~\eqref{E2.1.2}.
If
\begin{equation}
	\nu \ge - 1,
	\label{E2.1.5}
\end{equation}
then in accordance with Theorem~\ref{T2.1} any solution of~\eqref{E2.1.1} is trivial.

In the case of $n > 2 k$, condition~\eqref{E2.1.5} is exact. Really, let $\nu < - 1$. We put
$$
	w _0 (r) = (2 + r)^{- n} \log^{- (\nu + \lambda) / (\lambda - 1)} (2 + r)
$$
and
$$
	w_i (r)
	=
	\frac{1}{n - 2}
	\int_0^r
	\left(
		\frac{\rho}{r}
	\right)^{n - 2}
	\rho
	w_{i - 1} (\rho)
	\,
	d\rho
	+
	\frac{1}{n - 2}
	\int_r^\infty
	\rho
	w_{i - 1} (\rho)
	\,
	d\rho,
	\quad
	i = 1, \ldots, k.
$$
For all $1 \le i \le k$ we obviously have
$$
	- \left(
		\frac{d^2}{dr^2} + \frac{n - 1}{r} \frac{d}{dr}
	\right)
	w_i
	=
	w_{i - 1}
$$
and
$$
	w_i (r) \sim r^{2 i - n} \log^{- (1 + \nu) / (\lambda - 1)} r
	\quad
	\mbox{as } r \to \infty,
$$
i.e. 
$$
	c_1 r^{2 i - n} \log^{- (1 + \nu) / (\lambda - 1)} r
	\le
	w_i (r) 
	\le 
	c_2 r^{2 i - n} \log^{- (1 + \nu) / (\lambda - 1)} r
$$
with some constants $c_1 > 0$ and $c_2 > 0$ for all $r$ in a neighborhood of infinity.
Thus, taking
\begin{equation}
	u (x) = \varepsilon w_k (|x|),
	\label{E2.1.6}
\end{equation}
where $\varepsilon > 0$ is a sufficiently small real number, 
we obtain nontrivial solution of~\eqref{E2.1.1} with some positive function $b$
for which~\eqref{E2.1.4} holds.
\end{Example}

\begin{Example}\label{E2.2}
In~\eqref{E2.1.1}, let the positive measurable function $b$ satisfy the inequality
\begin{equation}
	b (x)
	\ge
	b_0
	\sum_{i = 1}^\infty
	\chi_{\omega_i} (x)
	|x|^{(n - 2 k) \lambda - n}
	\log^\nu |x|,
	\quad
	b_0 = const > 0,
	\label{E2.2.1}
\end{equation}
for almost all $x$ from a neighborhood of infinity, 
where $\chi_{\omega_i}$ is the characteristic function of the set 
$\omega_i = \{ x \in {\mathbb R}^n : 2^{2 i} < |x| < 2^{2 i + 1} \}$, i.e.
$$
	\chi_{\omega_i} (x)
	=
	\left\{
		\begin{aligned}
			&
			1,
			&
			&
			x \in \omega_i,
			\\
			&
			0,
			&
			&
			x \in {\mathbb R}^n \setminus \omega_i.
		\end{aligned}
	\right.
$$
In addition, let $\lambda > 1$ and~\eqref{E2.1.5} be valid. 
Then, applying Theorem~\ref{T2.1} with $\sigma = 2^{1 / 4}$, 
we obtain that any solution of~\eqref{E2.1.1} is trivial.
As in Example~\ref{E2.1}, condition~\eqref{E2.1.5} is exact.
Really, the right-hand side of~\eqref{E2.2.1} does not exceed the right-hand side of~\eqref{E2.1.4}.
Thus, in the case of $\nu < -1$, formula~\eqref{E2.1.6} provides us again with the desirable solution.

In general, let the positive function $b$ satisfy the inequality
$$
	b (x)
	\ge
	\sum_{i = 1}^\infty
	b_i
	\chi_{
		B_{2 r_i}
		\setminus
		B_{r_i}
	} 
	(x),
$$
where $b_i \ge 0$ and $r_i > 0$ are some real numbers 
with $r_{i + 1} \ge 2 r_i$, $i = 1, 2, \ldots$.
According to Theorem~\ref{T2.1}, if $\lambda > 1$ and
$$
	\sum_{i = 1}^\infty
	r_i^{(m - n) \lambda + n}
	b_i
	=
	\infty,
$$
then any solution of~\eqref{E2.1.1} is trivial.
\end{Example}

\section{Proof of Theorems~\ref{T2.1} and~\ref{T2.2}}\label{proofOfTheorems}

From now on we assume that $u$ is a solution of~\eqref{1.1} and, moreover, 
$\lambda > 1$ and $\sigma > 1$ are some real numbers.
By $C$ we mean various positive constants that can depend only on $n$, $m$, $\lambda$, and $\sigma$.
Let us denote
$$
	J (r)
	=
	\int_{B_r}
	b (x)
	|u|^\lambda
	\,
	dx
$$
and
$$
	h (r)
	=
	\left(
		\frac{1}{r^n}
		\int_{
			B_{\sqrt{\sigma} r}
			\setminus
			B_{r / \sqrt{\sigma}} 
		}
		a^{\lambda / (\lambda - 1)} (x)
		b^{- 1 / (\lambda - 1)} (x)
		\,
		dx
	\right)^{1 - \lambda},
$$
where $r$ runs over the set of positive real numbers.
As in the case of the function $q$, it is assumed that $h (r) = 0$ if
$$
	\int_{
		B_{\sqrt{\sigma} r}
		\setminus
		B_{r / \sqrt{\sigma}} 
	}
	a^{\lambda / (\lambda - 1)} (x)
	b^{- 1 / (\lambda - 1)} (x)
	\,
	dx
	=
	\infty.
$$

\begin{Lemma}\label{L3.1}
For arbitrary real numbers $0 < r_1 < r_2$ such that $r_2 \le \sqrt{\sigma} r_1$
the following relationship holds:
\begin{equation}
	J (r_2) - J (r_1)
	\ge
	C
	(r_2 - r_1)^{\lambda m}
	r_1^{(1 - \lambda) n}
	\sup_{(r_1, r_2)}
	h
	\,
	J^\lambda (r_1).
	\label{L3.1.1}
\end{equation}
\end{Lemma}

\begin{proof}
Let $\varphi_0 \in C^\infty ({\mathbb R})$ be a non-decreasing function satisfying the conditions
$$
	\left.
		\varphi_0
	\right|_{
		(- \infty, 0]
	}
	=
	0
	\quad
	\mbox{and}
	\quad
	\left.
		\varphi_0
	\right|_{
		[1, \infty)
	}
	=
	1.
$$
We put
$$
	\varphi (x)
	=
	\varphi_0
	\left(
		\frac{r_2 - |x|}{r_2 - r_1}
	\right).
$$
It is easy to see that
$$
	\int_{{\mathbb R}^n}
	b (x) |u|^\lambda
	\varphi
	\,
	dx
	\ge
	\int_{
		B_{r_1}
	}
	b (x)
	|u|^\lambda
	\,
	dx
$$
and
$$
	\left|
		\int_{{\mathbb R}^n}
		\sum_{|\alpha| = m}
		a_\alpha (x, u)
		\partial^\alpha
		\varphi
		\,
		dx
	\right|
	\le
	\frac{
		C
	}{
		(r_2 - r_1)^m
	}
	\int_{
		B_{r_2} \setminus B_{r_1}
	}
	a (x)
	|u|
	\,
	dx.
$$
Combining these estimates with~\eqref{1.3}, we obtain
\begin{equation}
	\int_{
		B_{r_2} \setminus B_{r_1}
	}
	a (x)
	|u|
	\,
	dx
	\ge
	C
	(r_2 - r_1)^m
	\int_{
		B_{r_1}
	}
	b (x)
	|u|^\lambda
	\,
	dx.
	\label{PL3.1.1}
\end{equation}
If
$$
	\int_{
		B_{r_2}
		\setminus
		B_{r_1} 
	}
	a^{\lambda / (\lambda - 1)} (x)
	b^{- 1 / (\lambda - 1)} (x)
	\,
	dx
	=
	\infty,
$$
then~\eqref{L3.1.1} is obvious since $h (r) = 0$ for all $r \in (r_1, r_2)$.
Hence, we have to consider only the case where
$$
	\int_{
		B_{r_2}
		\setminus
		B_{r_1} 
	}
	a^{\lambda / (\lambda - 1)} (x)
	b^{- 1 / (\lambda - 1)} (x)
	\,
	dx
	<
	\infty.
$$
In this case, by the H\"older inequality, we have
\begin{align*}
	\int_{
		B_{r_2} \setminus B_{r_1}
	}
	a (x)
	|u|
	\,
	dx
	\le
	{}
	&
	\left(
		\int_{
			B_{r_2}
			\setminus
			B_{r_1} 
		}
		a^{\lambda / (\lambda - 1)} (x)
		b^{- 1 / (\lambda - 1)} (x)
		\,
		dx
	\right)^{(\lambda - 1) / \lambda}
	\\
	{}
	&
	\times
	\left(
		\int_{
			B_{r_2}
			\setminus
			B_{r_1} 
		}
		b (x)
		|u|^\lambda
		\,
		dx
	\right)^{1 / \lambda}.
\end{align*}
According to~\eqref{PL3.1.1}, this implies the estimate
\begin{align*}
	\int_{
		B_{r_2}
		\setminus
		B_{r_1} 
	}
	b (x)
	|u|^\lambda
	\,
	dx
	\ge
	{}
	&
	C
	(r_2 - r_1)^{\lambda m}
	\left(
		\int_{
			B_{r_2}
			\setminus
			B_{r_1} 
		}
		a^{\lambda / (\lambda - 1)} (x)
		b^{- 1 / (\lambda - 1)} (x)
		\,
		dx
	\right)^{1 - \lambda}
	\\
	{}
	&
	\times
	\left(
		\int_{
			B_{r_1}
		}
		b (x)
		|u|^\lambda
		\,
		dx
	\right)^\lambda
\end{align*}
which yields relationship~\eqref{L3.1.1} immediately.
\end{proof}

\begin{Lemma}\label{L3.2}
Let $0 < r_1 < r_2$ be some real numbers such that
$r_2 \le \sqrt{\sigma} r_1$ and $2 J (r_1) \le J (r_2)$.
If $J (r_1) > 0$, then
$$
	J^{(1 / \lambda - 1) / m} (r_1)
	-
	J^{(1 / \lambda - 1) / m} (r_2)
	\ge
	C
	\int_{r_1}^{r_2}
	r^{(1 / \lambda - 1) n / m}
	h^{1 / (\lambda m)} (r)
	\,
	dr.
$$
\end{Lemma}

\begin{proof}
We construct a finite sequence of real numbers $\{ \rho_i \}_{i=0}^k$.
Let $\rho_0 = r_1$. Assume further that $\rho_i$ is already defined. 
If $4 J (\rho_i) \ge J (r_2)$, we put $r_{i + 1} = r_2$, $k = i + 1$ and stop.
Otherwise we take $\rho_{i+1} \in (\rho_i, r_2)$ such that $J (\rho_{i+1}) = 2 J (\rho_i)$.
Since $J$ is a continuous function, such a real number $\rho_{i+1}$ obviously exists.
It is also clear that this procedure must terminate at a finite step. 

From the construction of the sequence $\{ \rho_i \}_{i=0}^k$, it follows that
\begin{equation}
	r_1 \le \rho_i < \rho_{i+1} \le r_2
	\quad
	\mbox{and}
	\quad
	2 J (\rho_i) \le J (\rho_{i+1}) \le 4 J (\rho_i),
	\quad
	i = 0, \ldots, k - 1.
	\label{PL3.2.1}
\end{equation}

By Lemma~\ref{L3.1}, for all $0 \le i \le k - 1$ we obtain
$$
	J (\rho_{i+1}) - J (\rho_i)
	\ge
	C
	(\rho_{i+1} - \rho_i)^{\lambda m}
	\rho_i^{(1 - \lambda) n}
	\sup_{(\rho_i, \rho_{i+1})}
	h
	\,
	J^\lambda (\rho_i)
$$
or, in other words,
\begin{equation}
	\left(
		\frac{
			J (\rho_{i+1}) - J (\rho_i)
		}{
			J^\lambda (\rho_i)
		}
	\right)^{1 / (\lambda m)}
	\ge
	C
	(\rho_{i+1} - \rho_i)
	\rho_i^{(1 / \lambda - 1) n / m}
	\sup_{
		(\rho_i, \rho_{i+1})
	}
	h^{1 / (\lambda m)}.
	\label{PL3.2.2}
\end{equation}
Due to~\eqref{PL3.2.1} the following inequalities hold:
\begin{align}
	&
	J^{(1 / \lambda - 1) / m} (\rho_i)
	-
	J^{(1 / \lambda - 1) / m} (\rho_{i+1})
	\nonumber
	\\
	&
	{}
	\qquad
	\ge
	C
	\left(
		\frac{
			J (\rho_{i+1}) - J (\rho_i)
		}{
			J^\lambda (\rho_i)
		}
	\right)^{1 / (\lambda m)},
	\quad
	i = 0, \ldots, k - 1.
	\label{PL3.2.3}
\end{align}
Moreover, since $\rho \mapsto \rho^{(1 / \lambda - 1) / m}$
is a decreasing function on the interval $(0, \infty)$, it is obvious that
\begin{align}
	&
	(\rho_{i+1} - \rho_i)
	\rho_i^{(1 / \lambda - 1) n / m}
	\sup_{
		(\rho_i, \rho_{i+1})
	}
	h^{1 / (\lambda m)}
	\nonumber
	\\
	&
	{}
	\qquad
	\ge
	\int_{
		\rho_i
	}^{
		\rho_{i+1}
	}
	r^{(1 / \lambda - 1) n / m}
	h^{1 / (\lambda m)} (r)
	\,
	dr,
	\quad
	i = 0, \ldots, k - 1.
	\label{PL3.2.4}
\end{align}
Combining~\eqref{PL3.2.2}, \eqref{PL3.2.3}, and~\eqref{PL3.2.4}, we have
\begin{align*}
	&
	J^{(1 / \lambda - 1) / m} (\rho_i)
	-
	J^{(1 / \lambda - 1) / m} (\rho_{i+1})
	\\
	&
	{}
	\qquad
	\ge
	C
	\int_{
		\rho_i
	}^{
		\rho_{i+1}
	}
	r^{(1 / \lambda - 1) n / m}
	h^{1 / (\lambda m)} (r)
	\,
	dr,
	\quad
	i = 0, \ldots, k - 1.
\end{align*}
Finally, summing the last inequalities over all $0 \le i \le k - 1$, we complete the proof.
\end{proof}

\begin{Lemma}\label{L3.3}
Let $0 < r_1 < r_2$ be some real numbers such that
$r_2 = \sqrt{\sigma} r_1$ and $J (r_2) \le 2 J (r_1)$.
If $J (r_1) > 0$, then
$$
	J^{1 - \lambda} (r_1)
	- 
	J^{1 - \lambda} (r_2)
	\ge
	C
	\int_{r_1}^{r_2}
	r^{(m - n) \lambda + n - 1}
	h (r)
	\,
	dr.
$$
\end{Lemma}

\begin{proof}
Lemma~\ref{L3.1} implies the estimate
\begin{equation}
	(J (r_2) - J (r_1))
	J^{- \lambda} (r_1)
	\ge
	C
	(r_2 - r_1)^{\lambda m}
	r_1^{(1 - \lambda) n}
	\sup_{(r_1, r_2)}
	h.
	\label{PL3.3.1}
\end{equation}
By Lagrange's average value theorem, we obtain
$$
	J^{1 - \lambda} (r_1)
	- 
	J^{1 - \lambda} (r_2)
	=
	(\lambda - 1)
	(J (r_2) - J (r_1))
	R^{- \lambda}.
$$
for some real number $R \in (J (r_1), J (r_2))$.
Since $R \le 2 J (r_1)$, this allows us to assert that
$$
	J^{1 - \lambda} (r_1)
	- 
	J^{1 - \lambda} (r_2)
	\ge
	C
	(J (r_2) - J (r_1))
	J^{- \lambda} (r_1).
$$
Taking into account the condition $r_2 = \sqrt{\sigma} r_1$, we also have
$$
	(r_2 - r_1)^{\lambda m}
	r_1^{(1 - \lambda) n}
	\sup_{(r_1, r_2)}
	h
	\ge
	C
	\int_{r_1}^{r_2}
	r^{(m - n) \lambda + n - 1}
	h (r)
	\,
	dr.
$$
In virtue of~\eqref{PL3.3.1} the last two inequalities readily yield the desirable estimate.
\end{proof}

\begin{Lemma}[see~{\cite[Lemma~2.6]{meIzv}}]\label{L3.4}
Let
$0 < \alpha \le 1$,
$\varkappa > 1$,
$\nu > 1$,
$r_1 > 0$,
and
$r_2 > 0$
be real numbers with
$r_2 \ge \nu r_1$.
Then
$$
	\left(
		\int_{r_1}^{r_2}
		\psi^\alpha (r)
		\,
		dr
	\right)^{1 / \alpha}
	\ge
	A
	\int_{r_1}^{r_2}
	r^{1 / \alpha - 1}
	\gamma (r)
	\,
	dr
$$
for any measurable function $\psi : [r_1, r_2] \to [0,\infty)$,
where 
$$
	\gamma (r)
	=
	\essinf_{
		(r / \varkappa, \varkappa r)
	}
	\psi
$$
and $A > 0$ is a constant depending only on $\alpha$, $\varkappa$, and $\nu$.
\end{Lemma}

\begin{proof}[Proof of Theorem~$\ref{T2.2}$]
Assume the converse. Then there is a real number $r_0 > 0$ such that $J (r_0) > 0$.
Let us put $r_i = \sigma^{i/2} r_0$, $i = 1,2,\ldots$.
We denote by $\Xi_1$ the set of non-negative integers $i$ for which $J (r_{i + 1}) \ge 2 J (r_i)$. 
In so doing, let $\Xi_2$ be the set of all other non-negative integers.

If $i \in \Xi_1$, then in accordance with Lemma~\ref{L3.2} we obtain
\begin{equation}
	J^{(1 / \lambda - 1) / m} (r_i)
	-
	J^{(1 / \lambda - 1) / m} (r_{i + 1})
	\ge
	C
	\int_{
		r_i
	}^{
		r_{i + 1}
	}
	r^{(1 / \lambda - 1) n / m}
	h^{1 / (\lambda m)} (r)
	\,
	dr.
	\label{PT2.2.1}
\end{equation}
In turn, if $i \in \Xi_2$, then Lemma~\ref{L3.3} implies the estimate
\begin{equation}
	J^{1 - \lambda} (r_i)
	- 
	J^{1 - \lambda} (r_{i + 1})
	\ge
	C
	\int_{
		r_i
	}^{
		r_{i + 1}
	}
	r^{(m - n) \lambda + n - 1}
	h (r)
	\,
	dr.
	\label{PT2.2.2}
\end{equation}

Let us suppose that
\begin{equation}
	\sum_{i \in \Xi_1}
	\int_{
		r_i
	}^{
		r_{i + 1}
	}
	r^{(m - n) \lambda + n - 1}
	q (r)
	\,
	dr
	=
	\infty,
	\label{PT2.2.3}
\end{equation}
where $q$ is defined by~\eqref{T2.2.1}.
Then, summing inequalities~\eqref{PT2.2.1} over all $i \in \Xi_1$, we obtain
$$
	J^{(1 / \lambda - 1) / m} (r_0)
	\ge
	C
	\sum_{i \in \Xi_1}
	\int_{
		r_i
	}^{
		r_{i + 1}
	}
	r^{(1 / \lambda - 1) n / m}
	h^{1 / (\lambda m)} (r)
	\,
	dr,
$$
whence in virtue of $\lambda m > 1$ it follows that
\begin{align}
	J^{1 - \lambda} (r_0)
	&
	\ge
	C
	\left(
		\sum_{i \in \Xi_1}
		\int_{
			r_i
		}^{
			r_{i + 1}
		}
		r^{(1 / \lambda - 1) n / m}
		h^{1 / (\lambda m)} (r)
		\,
		dr
	\right)^{\lambda m}
	\nonumber
	\\
	&
	\ge
	C
	\sum_{i \in \Xi_1}
	\left(
		\int_{
			r_i
		}^{
			r_{i + 1}
		}
		r^{(1 / \lambda - 1) n / m}
		h^{1 / (\lambda m)} (r)
		\,
		dr
	\right)^{\lambda m}.
	\label{PT2.2.4}
\end{align}
Using Lemma~\ref{L3.4} with 
$\psi (r) = r^{(1 - \lambda) n} h (r)$, 
$\alpha = 1 / (\lambda m)$, 
and $\varkappa = \nu = \sqrt{\sigma}$, 
we obtain
$$
	\left(
		\int_{
			r_i
		}^{
			r_{i + 1}
		}
		r^{(1 / \lambda - 1) n / m}
		h^{1 / (\lambda m)} (r)
		\,
		dr
	\right)^{\lambda m}
	\ge
	C
	\int_{
		r_i
	}^{
		r_{i + 1}
	}
	r^{(m - n) \lambda + n - 1}
	q (r)
	\,
	dr
$$
for all $i \in \Xi_1$.
Hence,~\eqref{PT2.2.4} implies the inequality
$$
	J^{1 - \lambda} (r_0)
	\ge
	C
	\sum_{i \in \Xi_1}
	\int_{
		r_i
	}^{
		r_{i + 1}
	}
	r^{(m - n) \lambda + n - 1}
	q (r)
	\,
	dr
$$
whose right-hand side is equal to infinity while the left-hand side is bounded.
We obviously arrive at a contradiction.
Therefore,~\eqref{PT2.2.3} can not be valid, and due to condition~\eqref{T2.1} we have
\begin{equation}
	\sum_{i \in \Xi_2}
	\int_{
		r_i
	}^{
		r_{i + 1}
	}
	r^{(m - n) \lambda + n - 1}
	q (r)
	\,
	dr
	=
	\infty.
	\label{PT2.2.5}
\end{equation}
Summing inequalities~\eqref{PT2.2.2} over all $i \in \Xi_2$, one can conclude that
$$
	J^{1 - \lambda} (r_0)
	\ge
	C
	\sum_{i \in \Xi_2}
	\int_{
		r_i
	}^{
		r_{i + 1}
	}
	r^{(m - n) \lambda + n - 1}
	h (r)
	\,
	dr.
$$
Since $h (r) \ge q (r)$ for all $r \in (0, \infty)$, this yields
$$
	J^{1 - \lambda} (r_0)
	\ge
	C
	\sum_{i \in \Xi_2}
	\int_{
		r_i
	}^{
		r_{i + 1}
	}
	r^{(m - n) \lambda + n - 1}
	q (r)
	\,
	dr.
$$
In virtue of~\eqref{PT2.2.5} the last inequality again leads us to a contradiction.
\end{proof}

\begin{proof}[Proof of Theorem~$\ref{T2.1}$]
We note that
$$
	\left(
		\frac{1}{r^n}
		\int_{
			B_{\sigma r}
			\setminus
			B_{r / \sigma} 
		}
		a^{\lambda / (\lambda - 1)} (x)
		b^{- 1 / (\lambda - 1)} (x)
		\,
		dx
	\right)^{1 - \lambda}
	\ge
	C
	\essinf_{
		B_{\sigma r}
		\setminus
		B_{r / \sigma} 
	}
	a^{- \lambda} b
$$
for all $r \in (0, \infty)$. Thus, Theorem~\ref{T2.1} follows immediately from Theorem~\ref{T2.2}.
\end{proof}

\section{Generalizations}

The condition of strict positivity of the function $a$ can be omitted. Instead, we assume that
$$
	b (x) \ge a^\lambda (x) f (x) 
$$
with some non-negative function $f \in L_{\infty, loc} ({\mathbb R}^n)$
for almost all $x \in {\mathbb R}^n$.

\begin{Theorem}\label{T4.1}
Let $\lambda > 1$ and~\eqref{T2.1.1} be valid, where
$$
	q (r)
	=
	\essinf_{
		B_{\sigma r}
		\setminus
		B_{r / \sigma} 
	}
	f
$$
for some real number $\sigma > 1$.
Then any solution of~\eqref{1.1} is trivial.
\end{Theorem}

\begin{Theorem}\label{T4.2}
Let $\lambda > 1$ and~\eqref{T2.1.1} be valid, where
$$
	q (r)
	=
	\left(
		\frac{1}{r^n}
		\int_{
			B_{\sigma r}
			\setminus
			B_{r / \sigma} 
		}
		f^{- 1 / (\lambda - 1)} (x)
		\,
		dx
	\right)^{1 - \lambda}
$$
for some real number $\sigma > 1$.
Then any solution of~\eqref{1.1} is trivial.
\end{Theorem}

It can be seen that
$$
	\left(
		\frac{1}{r^n}
		\int_{
			B_{\sigma r}
			\setminus
			B_{r / \sigma} 
		}
		f^{- 1 / (\lambda - 1)} (x)
		\,
		dx
	\right)^{1 - \lambda}
	\ge
	C
	\essinf_{
		B_{\sigma r}
		\setminus
		B_{r / \sigma} 
	}
	f
$$
for all $r \in (0, \infty)$. Thus, Theorem~\ref{T4.1} is a consequence of Theorem~\ref{T4.2}. 
In turn, to prove Theorem~\ref{T4.2}, 
it is sufficient to establish the validity of Lemma~\ref{L3.1} with
$$
	h (r)
	=
	\left(
		\frac{1}{r^n}
		\int_{
			B_{\sqrt{\sigma} r}
			\setminus
			B_{r / \sqrt{\sigma}} 
		}
		f^{- 1 / (\lambda - 1)} (x)
		\,
		dx
	\right)^{1 - \lambda}.
$$
Really, in the case of
$$
	\int_{
		B_{r_2} \setminus B_{r_1}
	}
	f^{- 1 / (\lambda - 1)} (x)
	\,
	dx
	=
	\infty,
$$
Lemma~\ref{L3.1} is evident since $h (r) = 0$ for all $r \in (r_1, r_2)$.
If
$$
	\int_{
		B_{r_2} \setminus B_{r_1}
	}
	f^{- 1 / (\lambda - 1)} (x)
	\,
	dx
	<
	\infty,
$$
then~\eqref{PL3.1.1} implies the estimate
\begin{equation}
	\int_{
		B_{r_2} \setminus B_{r_1}
	}
	f^{- 1 / \lambda} (x)
	b^{1 / \lambda} (x)
	|u|
	\,
	dx
	\ge
	C
	(r_2 - r_1)^m
	\int_{
		B_{r_1}
	}
	b (x)
	|u|^\lambda
	\,
	dx.
	\label{4.1}
\end{equation}
By the H\"older inequality, we obtain
\begin{align*}
	\int_{
		B_{r_2} \setminus B_{r_1}
	}
	f^{- 1 / \lambda} (x)
	b^{1 / \lambda} (x)
	|u|
	\,
	dx
	\le
	{}
	&
	\left(
		\int_{
			B_{r_2} \setminus B_{r_1}
		}
		f^{- 1 / (\lambda - 1)} (x)
		\,
		dx
	\right)^{(\lambda - 1) / \lambda}
	\\
	{}
	&
	\times
	\left(
		\int_{
			B_{r_2} \setminus B_{r_1}
		}
		b (x)
		|u|^\lambda
		\,
		dx
	\right)^{1 / \lambda},
\end{align*}
whence in accordance with~\eqref{4.1} it follows that
\begin{align*}
	\int_{
		B_{r_2}
		\setminus
		B_{r_1} 
	}
	b (x)
	|u|^\lambda
	\,
	dx
	\ge
	{}
	&
	C
	(r_2 - r_1)^{\lambda m}
	\left(
		\int_{
			B_{r_2}
			\setminus
			B_{r_1} 
		}
		f^{- 1 / (\lambda - 1)} (x)
		\,
		dx
	\right)^{1 - \lambda}
	\\
	{}
	&
	\times
	\left(
		\int_{
			B_{r_1}
		}
		b (x)
		|u|^\lambda
		\,
		dx
	\right)^\lambda.
\end{align*}
This obviously proves~\eqref{L3.1.1}.

Finally, note that, instead of~\eqref{1.2}, we can consider the inequality
$$
	|a_\alpha (x, \zeta)| \le a (x) |\zeta|^p,
$$
where $p > 0$ is some real number, as was done in~\cite{MPbook}.
However, this does not increase the generality since
we can always use the replacement $v =  |u|^p \sign u$.

\end{document}